\documentclass[11pt, a4paper]{amsart}
\usepackage{amsmath,amssymb,amscd,amsfonts, graphicx, verbatim}

\numberwithin{equation}{section}
\def \seq {\sim_s}

\newcommand{\R}{{\mathbb R}}

\newcommand{\N}{{\mathbb N}}
\newcommand{\ol}{\overline}
\newcommand{\wt}{\widetilde}

\newcommand{\rk}{\operatorname {rk}}
\newcommand{\calH}{\mathcal H}

\newtheorem{teo}{Theorem }[section]
\newtheorem{prop}[teo]{Proposition }
\newtheorem{rem}[teo]{Remark }
\newtheorem{notation}[teo]{Notation }
\newtheorem{lem}[teo]{Lemma } 
\newtheorem{cor}[teo]{Corollary }
\newtheorem{defi}[teo]{Definition }
\newtheorem{ex}[teo]{Example }

\begin{document}
\title{Algebraic approximation preserving dimension}

\author{M. Ferrarotti} 
\address{Dipartimento di Matematica\\Politecnico di Torino\\Corso Duca
degli Abruzzi 24\\I-10129 Torino, Italy}
\email{ferrarotti@polito.it}
\thanks{This research was partially supported by M.I.U.R. and by G.N.S.A.G.A}

\author{E. Fortuna}
\address{Dipartimento di Matematica\\Universit\`a di Pisa\\Largo
  B. Pontecorvo 5\\I-56127 Pisa, Italy}
\email{fortuna@dm.unipi.it}

\author{L. Wilson}
\address{Department of Mathematics\\University of Hawaii, Manoa\\Honolulu, HI 96822, USA}
\email{les@math.hawaii.edu}






\begin{abstract}
We prove that each semialgebraic subset of $\R^n$ of positive codimension 
can be locally approximated of any order by means of an algebraic set 
of the same dimension.  As a consequence of previous results, algebraic 
approximation preserving dimension holds also for semianalytic sets.

\end{abstract}

\maketitle 

\section{Introduction}

If $A$ and $B$ are two closed subanalytic subsets of $\R^n$, the Hausdorff distance 
between their intersections with the sphere of radius $r$ centered 
at a common point $P$ can be used  to ``measure'' how 
near the two sets  are at $P$.  We say that $A$ and $B$ are 
$s$--equivalent (at $P$) if the previous  distance tends to $0$ more rapidly than $r^s$ 
(if so, we write $A \sim_sB$).

In the papers  \cite{FFW-SNS}, \cite{FFW-germs} and  \cite{FFW-semi} we addressed 
the question of the existence of an algebraic representative $Y$ in the class of 
$s$--equivalence of a given subanalytic set $A$ at a fixed point $P$. In this case we also 
say that $Y$ $s$-approximates $A$.

The answer to the previous  question is in general negative for subanalytic sets 
(see \cite{FFW-germs}). 

On the other hand  in \cite{FFW-SNS} it was proved that,  
for any real number $s\geq 1$ and for any closed semialgebraic set $A \subset \R^n$ of
codimension $\geq 1$, there exists an algebraic subset
$Y$ of $\R^n$ such that $A \sim_sY$.
The proof of the latter result consists in finding equations for $Y$ starting from the 
polynomials appearing in a presentation 
of $A$. For instance if $A=\{x\in \R^n\ |\ f(x)=0, h(x)\geq 0\}$ with 
$f,h\in \R[x]$, then  $A$ can be $s$-approximated by the algebraic set 
$Y=\{x\in \R^n\ |\ (f^2- h^m)(x)= 0\}$ for any odd integer $m$ sufficiently 
large. This procedure does not guarantee that $Y$ has the same dimension as $A$ at $P$ 
as the following trivial example shows.

Let $A$ be the positive $x_3$-axis in $\R^3$ presented 
as $A=\{(x_1,x_2,x_3)\in \R^3\ |\ x_1^2+x_2^2=0, x_3\geq 0\}$. Then according to the 
previous procedure, for any sufficiently large odd integer $m$,  $A$ is $s$-approximated 
at the origin $O$ by the algebraic set 
$Y=\{(x_1,x_2,x_3)\in \R^3\ |\ (x_1^2+x_2^2)^2 - x_3^m= 0\}$, 
whose germ at $O$ has dimension $2$. 
However we can also $s$-approximate $A$ at $O$ by the 
$1$-dimensional algebraic set 
$W=\{(x_1,x_2,x_3)\in \R^3\ |\ x_1^2- x_3^m= 0,x_2=0 \}$ 
for any sufficiently large odd  integer $m$. 
This algebraic set can be obtained by a similar construction as before but starting from 
the different presentation $A =\{(x_1,x_2,x_3)\in \R^3\ |\ x_1=0, x_2=0, x_3\geq 0\}$.

In \cite{FFW-semi} we proved that, for any $s \geq 1$,  
any closed semianalytic subset $A\subseteq \R^n$ is $s$-equivalent to 
a semialgebraic set $Y \subset \R^n$ having the same local dimension as $A$. 
However the arguments used in the proof of this latter result do not guarantee that, 
even if $A$ is analytic, it can be approximated  by means of an algebraic one 
of the same dimension.

In this paper we prove in Theorem \ref{preserve-dim}  that any semialgebraic set of
codimension $\geq 1$ is $s$-equivalent to an algebraic one of the same dimension. 
Using  the mentioned result of \cite{FFW-semi}, we obtain 
(Corollary \ref{preserve-semianal-dim}) that 
any semianalytic  set of codimension $\geq 1$ can be  $s$-approximated by an 
algebraic one preserving the local dimension. The proof of Theorem \ref{preserve-dim} 
works  provided that the semialgebraic set is described by means of 
a suitable presentation, as in
the previous example. Therefore Section \ref{reg-pres} is devoted to introduce the notion 
of ``regular presentation"  and to prove that  one can reduce to work with 
regularly presented sets.

\section{Basic properties of $s$-equivalence}

In this section we recall the definition and 
some basic properties of $s$-equivalence of subanalytic sets at a common point which, 
without loss of genericity, we can assume to be the origin $O$ of $\R^n$. We refer 
the reader to \cite{FFW-germs} for the proofs of these results.

If $A$ and $B$ are non-empty compact subsets of $\R^n$, 
let $\delta (A,B) = \sup _{x\in B} d(x,A) $. Thus, denoting by $D(A,B)$ 
the classical Hausdorff distance between the two sets, we have that 
$D(A,B)= \max  \{\delta (A,B),\ \delta (B,A)\}$.

\begin{defi} Let $A$ and $B$ be closed subanalytic subsets 
of $\R^n$ with $O\in A\cap B$. Let $s$ be a real number $\geq 1$. Denote by $S_r$ the 
sphere of radius $r$ centered at the origin. 
\begin{enumerate}
\item[(a)]  We say that $A\leq_s B$  if one of the following conditions holds:
\begin{enumerate}
\item[(i)] $O$ is isolated  in $A$,
\item[(ii)] $O$ is non-isolated both in $A$ and in $B$ and 
$$\lim_{r\to 0}\frac{\delta(B \cap S_r,A\cap S_r)}{r^s} =0.$$ 
\end {enumerate}
\item[( b)] We say that $A$ and
$B$ are $s$--equivalent  (and we will write $A \sim_sB$) if 
$A\leq_s B$ and $B\leq_s A$. 
\end{enumerate}
\end{defi}

Observe that, if $A\subseteq B$, then $A \leq_sB$ for any $s\geq 1$. 
It is easy to check that $\leq_s$ is transitive and that $\sim_s$ is an
equivalence relation. 

The following result shows the behavior of $s$-equivalence  
with respect to the union of sets:

\begin{prop}\label{union}  Let $A$, $A'$, $B$ and $B'$ be closed
subanalytic subsets of $\R^n$.
\begin{enumerate}
\item If $A \leq_s B$ and $A' \leq_s B'$, then $A\cup A' \leq_s B\cup B'$.
\item If $A \sim_s B$ and $A' \sim_s B'$, then $A\cup A' \sim_s B\cup B'$.
\end{enumerate}
\end{prop}

A useful tool to test the $s$-equivalence of two subanalytic sets is introduced in 
the following definition:

\begin{defi}  Let $A$ be  a closed  subanalytic subset
of $\R^n$, $O\in A$. For any real  $ \sigma>1$, 
we will call  \emph{horn-neighbourhood} with center $A$ and exponent $ \sigma$ the set
$$\mathcal H(A,\sigma)= \{ x\in \R^n \ |\  d(x,A)<  \|x\|^\sigma\}.$$ 
\end{defi}

\begin{rem}\label{semihorn} {\rm If $A$ is a closed semialgebraic subset of $\R^n$ and 
$\sigma$ is a rational number, then $\mathcal H(A,\sigma)$ is semialgebraic. Moreover 
if $O$ is isolated in $A$, then $\mathcal H(A,\sigma)$ is empty near $O$.
}\qed
\end{rem}

\begin{prop}\label{horn-lemma} Let $A,B$ be closed subanalytic subsets
of $\R^n$ with $O\in A\cap B$ and let $s \geq 1$. Then
$A\leq_s B$ if and only if there exist real constants $R>0$ and $\sigma >s$  
such that $$(A \setminus \{O\}) \cap B(O,R)\subseteq 
\mathcal H(B,\sigma)$$
where $B(O,R)$ denotes the open ball centered at $O$ of radius $R$.  
\end{prop}

The following technical result shows that it is possible to modify 
a subanalytic set by means of a suitable horn-neighborhood producing a new  
subanalytic set $s$-equivalent to the original one:

\begin{lem}\label{XY}  Let $X\subset Y \subset \R^n$ be closed subanalytic sets such that $O\in X$ and let $s\ge 1$. Then:
\begin{enumerate}
\item for any $\sigma >s$ we have $Y \seq Y \cup \calH (X,\sigma)$;
\item if $\ol{Y \setminus X} =Y$, there exists $\sigma >s$ such that 
$Y \setminus \calH (X,\sigma) \seq Y.$
\end{enumerate}
\end{lem}

Another essential tool will be the following version of \L ojasiewicz' inequality,  
proved in \cite{FFW-semi}; henceforth for any map $f \colon \R^n \to \R^p$ 
we will denote by $V(f)$ the zero-set $f^{-1}(O)$.

\begin{prop}\label{Loj}
Let $A$ be a compact subanalytic subset of
$\R^n$. Assume $f$ and $g$ are subanalytic functions defined on $A$
such that $f$ is continuous, $V(f) \subseteq V(g)$, $g$  is continuous at 
the points of $V(g)$ and such that $\sup |g|<1$. Then there exists
a positive constant $\alpha$ such that $|g|^\alpha \leq |f|$ on $A$ and $|g|^\alpha <
|f|$ on $A\setminus V(f)$.
\end{prop}

\section{Presentations of semialgebraic sets}\label{reg-pres}

This section is devoted to the first crucial step in our strategy, that is reducing 
ourselves  to prove the main theorem for semialgebraic sets suitably presented.

\begin{defi}\label{good}
Let $A$ be a closed semialgebraic subset of $\R^n$ with 
 $\dim_O A=d>0$.  We will say that 
$A$ admits a \emph{good presentation}  if 
 \begin{enumerate}
 \item[(a)]   the Zariski closure $\ol A^Z$   of  $A$  is irreducible 
 \item[(b)]   there exist generators $ f_1, \ldots, f_p$  of  the ideal $I(\ol A^Z)\subseteq 
 \R[x_1, \ldots, x_n]$ 
  and $h_1, \dots , h_q$ polynomial functions  such that 
$$A=\{x \in \R^n \ |\ f_i(x)=0, h_j(x)\geq 0, i=1, \ldots, p, j=1,\dots,q\}$$
\item[(c)]   $h_i(O)=0$ and $\dim _O (V(h_i) \cap V(f) ) < d$, for each $i$, where 
$f=(f_1, \ldots ,f_p)$.
\end{enumerate}
\end{defi}

\begin{lem}\label{get-good}
Let $A$ be a closed semialgebraic subset of $\R^n$ with 
 $\dim_O A=d>0$. Then   there exist closed 
semialgebraic sets $\Gamma_1, \ldots, \Gamma_r,\Gamma '$ such that
\begin{enumerate}
\item $A= \left(\bigcup_{i=1}^r \Gamma_i\right) \cup \Gamma '$
\item for each $i$, $\dim_O \Gamma_i =d$, and $\dim_O \Gamma ' < d$
\item for each $i$, $\Gamma_i$ admits a good presentation.
\end{enumerate}
\end{lem}
\begin{proof}
Arguing as in \cite[Lemma 3.2]{FFW-semi}  in the semialgebraic setting, there exist 
semialgebraic sets $\Gamma_1, \ldots, \Gamma_r,\Gamma '$ fulfilling conditions (1) 
and (2) of the thesis and such that, for each $i$,  $\Gamma_i$ 
admits a presentation satisfying 
conditions (a) and (b) of Definition \ref{good}. In order to achieve also condition (c) 
it suffices to drop from the presentation of each $\Gamma _i$ all the  
inequalities $h_j(x)\geq 0$ such that $h_j$ vanishes identically on $\Gamma _i$.
\end{proof}

Since we are interested in preserving dimension, we will reduce ourselves to work with 
a set presented by as many polynomial equations as its  codimension and with 
the critical locus of the associated polynomial map nowhere dense.

\begin{notation} \label{sigmanot} Let $\Omega $ be an open subset of $\R^n$.
For any smooth $\varphi \colon \Omega \to \R^p$, denote  
$\Sigma_r(\varphi)=\{x\in \Omega \ |\ \rk\ d_x \varphi <r\}$ and $\Sigma(\varphi)=\Sigma_p(\varphi)$.
\end{notation}

\begin{defi}\label{regular}
Let $A$ be a closed semialgebraic subset of $\R^n$ with 
 $\dim_O A=d>0$. We will say that 
$A$ admits a \emph{regular presentation}  if  there exist a polynomial map 
$F\colon \R^n \to \R^{n-d}$ and polynomial functions 
$h_1, \ldots, h_q$  such that 
 \begin{enumerate}
 \item[(a)]   $A=\{x \in \R^n \ |\ F(x)=  0, h_j(x)\geq 0, \ j=1,\dots,q\}$, 
 \item[(b)]   $\dim _O (\Sigma (F) \cap A) <d$
 \item[(c)]    $h_i(O)=0$ and $\dim _O (V(h_i) \cap A ) < d$, for each $i$.
 \end{enumerate}
\end{defi}

A useful tool to pass from a good presentation to a regular one will be 
the following result (for a proof see for instance 
\cite[Proposition 7.7.10]{BCR}):

\begin{lem}\label{sign} Let $A$ be a closed semialgebraic subset of $\R^n$ and let 
$h,g$ polynomial functions on $\R^n$. Then there exist polynomial functions 
$\varphi, \psi$ with $\varphi >0$ and $\psi\geq 0$ such that 
\begin{enumerate}
\item $sign(\varphi h + \psi g) =sign (h)$ on $A$
\item $V(\psi) \subseteq \ol{V(h)\cap A}^Z$.
\end{enumerate}
\end{lem}

\begin{prop}\label{reduction} Let $A$ be a closed semialgebraic subset of $\R^n$ 
with  $\dim_O A=d>0$ which admits a good presentation. Let $s>1$. Then
there exists a closed semialgebraic subset  $\widetilde A$   of $\R^n$ with  
$\dim_O \widetilde A=d>0$ such that 
 \begin{enumerate}
 \item $\widetilde A$   admits a regular presentation
 \item $\widetilde A \seq A$. 
 \end{enumerate}
 \end{prop}
 
\begin{proof} 

By hypothesis, 
we have that 
$$A=\{x \in \R^n \ |\ f(x)=O, h_j(x)\geq 0, j=1,\dots,q\}$$ with $f=(f_1,\dots, f_p)$ 
such that $V(f)$ is irreducible, $V(f)=\ol A^Z$  and $f_1,\dots, f_p$ generate the ideal 
$I(V(f))$. 
In particular   $\dim _O(\Sigma_{n-d}(f) \cap V(f))<d$ \ 
(see for instance \cite[Definition 3.3.3]{BCR}). 

If $p=n-d$, we have the thesis with $\widetilde A=A$; thus let $p> n-d$. 

Denote by $\Pi$ the set of surjective linear maps from $\R^p$ to $\R^{n-d}$ and consider 
the smooth map $\Phi \colon (\R^n-V(f) ) \times \Pi \to \R^{n-d}$ defined by 
$\Phi (x, \pi) = (\pi \circ f) (x)$ for all $x\in \R^n-V(f) $ and $\pi \in \Pi$.

The map $\Phi$ is transverse to $\{O\}$: namely the partial Jacobian matrix of $\Phi$ with respect to the variables in $\Pi$ (considered as an open subset of 
$\R^{p(n-d)}$) is the $(n-d) \times p(n-d)$ matrix

$$ \bmatrix f(x)& O&O &\dots & O\\ O&f(x)&O&\dots &O  \\ \vdots \\O&O&O& \dots &f(x)\endbmatrix;$$
thus, for all $x\in \R^n-V(f)$ and for all $\pi \in \Pi$ the Jacobian matrix of
$\Phi$ has rank $n-d$. 

As a consequence, by a well-known result of singularity theory (see for instance \cite[Lemma 3.2]{Bruce-Kirk}), 
we have that the map $\Phi_{\pi} \colon  \R^n-V(f)   \to \R^{n-d}$ defined by 
$\Phi_{\pi} (x)= \Phi (x, \pi)=(\pi \circ f)(x)$ is transverse to $\{O\}$ for all $\pi$ outside a set $\Gamma \subset \Pi$ of measure zero and hence $\pi \circ f$ is a submersion on 
$V(\pi \circ f)\setminus V(f)$ for all such $\pi$.

Let $x\in V(f)$ be a point at which $f$ has rank $n-d$. 
Then there is an open dense set $U\subset \Pi$ such that  for all $\pi \in U$ the map 
$\pi \circ f$ is a submersion at  $x$, and hence off some subvariety of $V(f)$ 
of dimension less than $d$.

Thus, if we choose $\pi_0\in (\Pi\setminus \Gamma)\cap U$, the map 
$F= \pi_0 \circ f
\colon \R^n \to \R^{n-d}$ satisfies the following properties: 
 \begin{itemize}
\item $\dim_O V(F)=\dim_O V(f)=d$,
\item $\Sigma(F) \cap V(F)\subseteq V(f)\subseteq V(F)$, 
\item $\dim _O (\Sigma(F) \cap V(F))<d$.
\end{itemize}
\medskip

We want to show that there exist polynomials $h_i'$ such that 
\begin{itemize}
\item $A=\{x \in \R^n \ |\ f(x)=O, h_i'(x)\geq 0, i=1,\dots,q\}$
\item $\dim _O (V(F) \cap  \bigcup_{i=1}^q V(h_i'))<d. $
\end{itemize}

Namely for each  $i\in \{1, \ldots, q\}$ denote by $W_i$ the union 
of the irreducible components $Y$ of 
$V(F)$ such that $\dim _O (V(h_i)\cap Y) <d$ ;  let also 
$T_i= \ol{V(F)\setminus W_i}^Z$.  
Note that $V(f) \subseteq W_i$.

If we apply Lemma \ref{sign} 
choosing $h=h_i$ and $g=  \|f\|^2$ on $W_i$, then there exist $\varphi, \psi$ with 
$\varphi >0$ and $\psi\geq 0$ such that the function $h'_i = \varphi h_i + \psi \|f\|^2$ 
has the same sign as $h_i$ on $W_i$ 
and $V(\psi) \subseteq \ol{V(h_i)\cap W_i}^Z$. 
Then 
\begin{itemize}
\item $V(h_i')\cap W_i = V(h_i)\cap W_i$  
\item since $h_i' |_{T_i} =  (\psi \|f\|^2 )|_{T_i}$, then  $V(h_i')\cap T_i = 
(V(\psi) \cap T_i)  \cup (V(f) \cap T_i) \subseteq W_i \cap T_i$. 
\end{itemize}
Thus 
 $\dim _O ( V(h_i')\cap V(F)) <d$ for any $i$ and $$A=\{x \in \R^n \ |\ f(x)=O, h_i'(x)\geq 0, \ 
 i=1,\dots,q\}.$$

\medskip

For each $m \in \N$ denote 
\begin{equation}\label{Am}
\widetilde A_m=\{x \in \R^n \ |\ F(x)=0,  \|x\|^{2m} - \|f(x)\|^2 \geq 0, 
h_i'(x)\geq 0, \ i=1,\dots,q\}.\end{equation}

Since $A \subseteq  \widetilde A_m \subseteq V(F)$, then $\dim _O \widetilde A_m =d$.

We claim that there exists $m$ such that $\widetilde A_m \sim_s A$. 
Since $A \subseteq  \widetilde A_m$, we trivially have that $A \leq_s 
\widetilde A_m$ for any $m$. Thus it  is sufficient to prove that there exists 
$m$ such that $\widetilde A_m  \leq_s A$.
Namely, let $\Lambda =\{x \in \R^n \ |\ h_i'(x)\geq 0, i=1,\dots,q\}$.  
Since $V(\|f\|)\cap \Lambda=A
=V(d(x,A) )\cap \Lambda $, by Proposition \ref{Loj}  there exist a rational number 
$\tau$ and a real number $R>0$ such that 
$$d(x,A)^\tau < \|f(x)\| \qquad \forall x\in (\Lambda\setminus V(f))\cap B(O,R)=
(\Lambda\setminus A)\cap B(O,R).$$
Let $m >  s\tau$. Then $d(x,A) < \|f(x)\|^{\frac 1\tau} \leq \|x\| ^ {\frac m\tau}$ for all 
$x\in (\widetilde A_m\setminus A) \cap B(O,R)$. This implies that 
$(\widetilde A_m \setminus \{O\}) 
\cap B(O,R) \subseteq \mathcal H(A,\frac m\tau)$ 
and hence, by Proposition \ref{horn-lemma}, $\widetilde A_m  \leq_s A$. 

Up to increasing $m$, we can also assume that 
$\dim _O (V(F) \cap V(\|x\|^{2m} - \|f(x)\|^2 ))<d$ and hence  that (\ref{Am}) 
is a a regular presentation of $\widetilde A_m$.

It is thus sufficient to choose $m$ as above and  $\widetilde A =\widetilde A_m$. 

\end{proof}

\section{Main result}

Since $s$--equivalence depends only on the germs at $O$,
we are allowed to identify a subanalytic set
with a realization of its germ at the origin in a suitable ball $B(O,R)$. Henceforth 
we will even omit to explicitely indicate the intersection of our sets with $B(O,R)$; 
in particular, given two sets $U$ and $U'$,  when we write that $U \subseteq U'$ 
we mean that $U\cap B(O,R) \subseteq U'$ for a suitable real constant $R>0$.

\begin{teo}\label{preserve-dim}
 For any real number
$s\geq 1$ and for any closed semialgebraic set $A \subset \R^n$ of
codimension $\geq 1$ with $O\in A$, there exists an algebraic subset
$S$ of $\R^n$ such that $A \seq S$ and $\dim_O S= \dim_O A$.
\end{teo}

\begin{proof} 

We will prove the thesis by induction on $d= \dim _O A$.

If  $d=0$ the result holds trivially. So let $d\geq 1$ and assume that the result holds 
for all semialgebraic sets of dimension less that $d$.

By Lemma \ref{get-good},  there exist closed 
semialgebraic sets $\Gamma_1, \ldots, \Gamma_r,\Gamma '$ such that
\begin{enumerate}
\item $A= \left(\bigcup_{i=1}^r \Gamma_i\right) \cup \Gamma '$
\item for each $i$, $\dim_O \Gamma_i =d$ and $\Gamma_i$ admits a good presentation
\item $\dim_O \Gamma ' < d$.
\end{enumerate}

By Proposition \ref{union}, by Proposition \ref{reduction} and by the inductive hypothesis 
we can assume that $A$ is described by means of a regular presentation as
$$A=\{x \in \R^n \ |\ F_0(x)=O, h_j(x)\geq 0, j=1,\dots,q\}$$
with $F_0=(f_1,\dots, f_{n-d})$. 
We can assume $q\geq 1$, because otherwise there is nothing to prove.

We will use the following notation:
\begin{itemize}
\item $Z_i=\bigcup_{j=i+1}^q V(h_j)$ for $i=0, \ldots, q-1$, \ and $Z_q=\emptyset$,
\item $X= (\Sigma(F_0) \cup Z_0)\cap A$,
\item $\wt f =(f_2,\dots, f_{n-d})\colon \R^n \to \R^{n-d-1}$ and $V=V(\wt f)$,
\item $\Lambda _i=\{x \in \R^n \ |\ h_j(x)\geq 0, \ j=i+1,\dots,q\}$ for any $i=0, \ldots, q-1$, \  and 
$\Lambda _{q}= \R^n$.
\end{itemize}

Since the presentation of $A$ is regular,  we have that 
$$\dim _O(\Sigma(F_0) \cap A)<d \quad\mbox{  and }  \quad \dim _O(Z_0 \cap A) <d.$$

Let $X_1= X \cap \ol{A \setminus X} $ and $X_2= \ol{X \setminus X_1}$. Since  
$A= \ol{A \setminus X_1} \cup X_2$ and $\dim_O X_2 <d$, by the inductive hypothesis 
it suffices to prove the thesis for $\ol {A \setminus X_1}$.

In other words we can assume that $A =\ol{A \setminus X}$.
 As a consequence  Lemma \ref{XY}  shows that there exists a rational number
$\sigma >s$ such that, if $K= \R^n \setminus   \calH(X,\sigma)$ then 
$A\cap K \seq A$. 

\medskip

Let $g_0=f_1$. 
We will recursively construct polynomial functions $g_1 \ldots, g_q$ such that, if  
$F_i=(g_i, f_2, \ldots, f_{n-d})$, 
then for any $i= 0, \ldots, q$ the semialgebraic subset 
$$A_i=\{x \in \R^n \ |\ F_i(x)=0, h_j(x)\geq 0, j=i+1,\dots,q\}=V(g_i) 
\cap V \cap \Lambda_{i}$$ 
satisfies the following properties 
\bigskip

\begin{enumerate}
\item[P1(i):] $ \left\{ 
\begin{array}{ll}
 A_i\seq A_{i-1} \mbox{ and } A_i\cap K\seq A_{i-1}\cap K & \mbox{ if } i\in\{1, \ldots, q\}\\
  A_0\cap K \seq A_0 & \mbox{ if }  i=0
  \end{array} \right. $
\item[P2(i):] $Z_i \cap A_i\cap K \subseteq \{O\}$
\item[P3(i):] $\Sigma(F_i) \cap A_i  \cap K \subseteq \{O\}$.
\end{enumerate}

\bigskip

As proved above, the set $A_0=A$ satisfies the properties P1(0), P2(0) and P3(0). 
Thus assume that $0 \leq i\leq q-1$, assume that we have already constructed 
$A_i$ fulfilling the three previous properties and let us construct $g_{i+1} $ in such a way 
that $A_{i+1}$ satisfies properties P1(i+1), P2(i+1) and P3(i+1).
\medskip

For any positive integer $m$ let $g_{i+1} =g_i^2-h_{i+1}^m$.

We want to see that there exists $m_s\in \N$ such that for any odd integer 
$m\geq m_s$ the semialgebraic set $A_{i+1}= V(g_{i+1}) 
\cap V\cap\Lambda_{i+2} $ satisfies properties 
P1(i+1), P2(i+1) and P3(i+1).

Properties P2(i) and P3(i) evidently guarantee that $(A_i\cap K ) \cap 
\Sigma(F_i) \cup Z_i )\subseteq \{O\}$. We will need to stregthen 
this fact as follows

{\bf Claim}:
There exists $\beta>s$ such that $\calH (A_i \cap K ,\beta ) \cap 
\left(\Sigma(F_i) \cup Z_i\right) \subseteq \{O\}$.

Namely, let $\phi : \Sigma(F_i)  \cup Z_i \to \R$ be the function defined by 
$\phi (x) = d(x, A_i\cap K)$
for every $x\in \Sigma(F_i)\cup Z_i$. The function $\phi$ is
semialgebraic, continuous and, by the previous properties P2(i) and P3(i), 
$V(\phi )= A_i \cap K\cap \left(\Sigma(F_i) 
\cup Z_i\right)   \subseteq \{O\}$.
Hence by Proposition \ref{Loj}  there exists a rational positive number 
$\beta$ such that $d(x, A_i\cap K ) > \|x\|^\beta  $ for all 
$ x\in \left(\Sigma(F_i)\cup Z_i\right) \setminus \{O\}$; evidently 
we can assume that $\beta > s$.
By definition of horn-neighborhood no  $x\ne O$ can lie in  
$\calH (A_i\cap K, \beta ) \cap \left(\Sigma(F_i)\cup Z_i\right) $ 
which proves the Claim.
\medskip

In particular for each $j=i+1, \ldots ,q$ 
we have that 
\begin{equation}\label{sign-h-i}
h_j |_{B(x, \|x\|^\beta)} >0 \qquad \forall \, x\in A_i \cap K \setminus \{O\} .
\end{equation}

\smallskip 
\noindent \underline{Property P1(i+1)}. Consider the set 
$E=\R^n \setminus \calH (A_i\cap K , \beta)$. 

Evidently the closed semialgebraic set $W= \left(V  \cap \Lambda_{i+1} \cap K \cap E\right) \cap \{h_{i+1} \geq 0\}$ fulfills the condition 
$$V(g_i) \cap W= (A_i\cap K ) \cap  E=\{O\}.$$ 
Thus by Proposition \ref{Loj} there exists $m_1\in \N$ such that for any integer number
$m\geq m_1$ we have  $g_i(x)^2\geq
h_{i+1}(x)^m$ for all $x\in W$ and $g_i(x)^2> h_{i+1}(x)^m$ for all $x\in
W\setminus \{O\}$.
\medskip

If we take $m$ an odd integer $\geq m_1$, by construction $g_{i+1}=g_i^2-h_{i+1}^m$ is 
strictly positive on $W\setminus \{O\}$ and on 
$\{h_{i+1}<0\}$, hence $g_{i+1}$ is strictly positive on 
$(V \cap \Lambda_{i+1} \cap K \cap E) \setminus \{O\}$. 
Since $A_{i+1}=V(g_{i+1}) \cap V \cap \Lambda_{i+1}$, it follows that 
\begin{equation}\label{rel1}
A_{i+1} \cap K \subseteq (\R^n \setminus E)\cup \{O\}=
\calH (A_i \cap K , \beta) \cup \{O\}
\end{equation} 
and therefore, by Proposition \ref{horn-lemma}, that  $$A_{i+1}\cap K \leq_s
A_i\cap K. $$

\medskip
We want now  to prove that $A_i\cap K \leq_s A_{i+1}\cap K $.

Consider the set $B_i= V\cap \Lambda_{i}\supseteq A_i$. 

Assume at first that  $B_i \cap K \setminus \{O\}$ is connected and 
denote by $d_g$ the geodesic distance on $B_i \cap K $; denote 
also by $B_g (x_0, r)=\{y\in B_i \cap K   \ |\ 
d_g (y, x_0) <r\}$ the geodesic ball centered at $x_0 \in B_i \cap K $.

By \cite{Loj}, up to working in a suitable Euclidean ball $B(O,R)$, there exist constants 
$C>0$ and $ 0<\alpha \leq 1$
such that for any $y_1,y_2 \in B_i \cap K \cap B(O,R)$ we have that 
$$\|y_1-y_2\| \leq d_g(y_1,y_2) \leq C \|y_1-y_2\|^{\alpha}.$$

Therefore we have  
$$B_g (x_0, r) \subseteq B(x_0, r) \cap B_i \cap K
\subseteq B_g (x_0, C r^{\alpha}) \qquad\forall x_0 \in B_i \cap K \cap B(O,R)$$ 
for $r$ small enough.
Up to decreasing $R$ and $\alpha$ if necessary, we can assume that 
$C=1$. 

\medskip

Property P3(i)  implies that,  for any $x\in A_i \cap K$, we have 
that $\dim_x (A_i\cap K)=d$ and, since $\rk d_x (\wt f) =n-d-1$, that  
$\dim _x (B_i  \cap K )=d+1$.
 Hence $\ol{ (B_i\cap K)\setminus (A_i  \cap K)}= B_i\cap K$.
Thus Lemma \ref{XY} assures that 
there exists a closed semialgebraic subset
$K' \subseteq B_i\cap K$ such that 
$$A_i\cap K'=A_i\cap K\cap K'=\{O\} \quad \mbox{ and } \quad 
B_i\cap K \sim_{\frac{s+\beta}\alpha} K'.$$

Evidently 
$V(g_i) \cap K'  = V(g_i) \cap B_i \cap K' =
A_i \cap K' =\{O\}.$  
Thus by Proposition \ref{Loj} there exists $m_2\in \N$ such that for any integer number
$m\geq m_2$ we have  $g_i(x)^2\geq
h_{i+1}(x)^m$ for all $x\in K'  $ and 
$g_i(x)^2> h_{i+1}(x)^m$ for all $x\in K'  \setminus \{O\}$.

If we take $m$ an integer $\geq m_2$, by construction $g_{i+1}=g_i^2-h_{i+1}^m$ is 
strictly positive on $K' \setminus \{O\}$.

Let $x \in A_i \cap K\setminus \{O\}$. As $h_{i+1}(x)>0$, 
then $g_{i+1}(x)<0$. Since 
$B_i\cap K \sim_{\frac{s+\beta}\alpha} K '$,
there exist $\eta >\frac{s+\beta}\alpha$  and $z\in K'$ such that $\|x-z\|<\|x\|^{\eta}$ 
(and we can assume that $z\ne O$).

As $g_{i+1}$ is strictly positive on $K '\setminus \{O\}$, $g_{i+1}(z)>0$. Since 
$z\in B(x, \|x\|^\eta)$, then  $z \in  B_g (x, \|x\| ^{\eta \alpha})$. So, by the
Intermediate Value Theorem on $B_g\left(x, \|x\|^{\eta\alpha}\right)$, there exists 
$w \in B_g\left(x, \|x\|^{\eta\alpha}\right) \subseteq B\left(x, \|x\|^{\eta\alpha}\right)
\cap B_i\cap K $ such that $g_{i+1}(w)=0$. 

Moreover, as $\eta \alpha>\beta$, by (\ref{sign-h-i}) one has in particular that $h_j(w)>0$
for any $j=i+2, \ldots ,q$, which means that $w\in A_{i+1}\cap K $; hence 
$x \in \calH( A_{i+1}\cap K, \eta \alpha).$

We have thus proved that $A_i \cap K\setminus \{O\}  \subseteq 
\calH(A_{i+1}\cap K, \eta\alpha )$
and therefore, since $\eta \alpha>s$, that 
$$A_i \cap K \leq_s A_{i+1}\cap K $$
by Proposition \ref{horn-lemma}.

In the general case, if $B_i\cap K \setminus \{O\}$ is not connected, 
it is sufficient to perform the previous argument on each connected component $\Delta$ 
of $B_i \cap K \setminus \{O\}$, find an odd integer number $m_{\Delta}$ as above, 
and take $m_2= \max{m_{\Delta}}$.

Hence,  if we let $M= \max\{m_1, m_2\}$, then for any odd $m \geq M$ we have 
\begin{equation}\label{partial-first-step}
A_{i+1}  \cap K \seq A_i\cap K .
\end{equation}

In order to conclude the proof that  $A_{i+1}$ satisfies property P1(i+1) observe 
that evidently 
\begin{equation}\label{first-step}
A_{i+1} \supseteq A_{i+1}\cap K 
\seq A_i\cap K \seq A\cap K \seq A
\end{equation}
and thus in particular  $A_{i+1} \geq_s A$. 

For any $\sigma '$ with $\sigma>\sigma' >s$ we have that 
$(A_{i+1}\cap \ol{\calH(X, \sigma)}) \setminus \{O\}\subseteq \calH(X, \sigma')$; hence 
$A_{i+1}\cap \ol{\calH(X, \sigma)} \leq_s X 
\subseteq A$ and thus $A_{i+1}\cap \ol{\calH(X, \sigma)} \leq_s A$.
Moreover from (\ref{first-step}) $A_{i+1}\cap K  \leq_s A$. 
Since 
$$A_{i+1} = (A_{i+1} \cap K )\cup 
(A_{i+1}\cap \ol{\calH(X, \sigma)}),$$
by Proposition \ref{union} we have $A_{i+1} 
\leq_s A$ and thus $A_{i+1} \seq A$. 
By the inductive hypothesis 
we also get that $$A_{i+1}  \seq A_i $$  and so P1(i+1) is proved.

\bigskip
\noindent \underline{Property P2(i+1)}.
By (\ref{rel1}) and the previous Claim, we have that 
\begin{equation}\label{sigma}
A_{i+1}\cap K \cap ( \Sigma(F_i) \cup Z_i)\subseteq \{O\}.
\end{equation} 
Thus $h_j$  does not vanish on $A_{i+1}\cap K \setminus \{O\}$
for any $j\geq i+1$, which in particular proves that $A_{i+1}$ satisfies property P2(i+1).

\bigskip

\noindent \underline{Property P3(i+1)}.
In order to prove P3(i+1) consider the Jacobian matrix of 
$F_{i+1}=(g_{i+1},f_2, \dots, f_{n-d})$, i.e.

$$\begin{pmatrix}
2 g_i \nabla g_i -m\, h_{i+1}^{m-1} \nabla h_{i+1} \\ 
\nabla f_2\\ \vdots \\ \nabla f_{n-d}
\end{pmatrix}.$$  

Evaluating it on the points of $A_{i+1}$ we get the matrix
$$\begin{pmatrix}
h_{i+1}^{\frac m2}(2  \nabla g_i -m\, h_{i+1}^{\frac m2-1} \nabla h_{i+1} )\\ 
\nabla f_2\\ \vdots \\ \nabla f_{n-d}
\end{pmatrix}.$$  
Since, as seen above,  $h_{i+1}$ does not vanish on $A_{i+1}\cap K \setminus \{O\}$, 
$$\Sigma (F_{i+1}) \cap A_{i+1} \cap K = 
\{ x\in  A_{i+1} \cap K \ |\ 
(2  \nabla g_i -m\, h_{i+1}^{\frac m2-1} \nabla h_{i+1} ) \wedge  
\nabla f_2 \wedge \ldots \wedge \nabla f_{n-d} =0\}.$$

If we let $\varphi= 4 \|\nabla g_i \wedge \nabla f_2 \wedge\ldots \wedge \nabla f_{n-d} \|^2$
and $\psi = m^2 \|  \nabla h_{i+1} \wedge \nabla f_2 \wedge \ldots \wedge \nabla f_{n-d}\|^2$
we have that 
$$\Sigma (F_{i+1}) \cap A_{i+1}\cap K = 
\{ x\in  A_{i+1} \cap K \ |\ \varphi (x) = |h_{i+1}(x)| ^{m-2} \psi(x)\}.$$

Since $V(\varphi) =\Sigma(F_i)$ and 
by (\ref{sigma}) $V(\varphi) \cap A_{i+1} \cap K \subseteq  \{O\}$, by Proposition 
\ref{Loj} there exists $\lambda$ such that $\varphi(x) \geq \|x\|^\lambda$ on  
$A_{i+1} \cap K $.
For the same reason there exists $\mu$ such that $|h_{i+1}(x)|^\mu \leq \|x\|$ on  
$A_{i+1} \cap K $. Moreover there exists a constant $N$ such that 
$\psi \leq N$  on  $A_{i+1} \cap K $.

If $m> \lambda \mu +2$, then $\Sigma (F_{i+1}) \cap A_{i+1} \cap K 
\subseteq \{O\}$. 
Namely, if by contradiction  there exists a sequence of points  
$x_\nu\in A_{i+1} \cap K$ 
converging to $O$ such that $\varphi (x_\nu) = |h_{i+1}(x_\nu)| ^{m-2} \psi(x_\nu)$, then 
$$\|x_\nu\|^{\lambda \mu}\leq N^\mu  \|x_\nu\|^ {m-2}$$
which is a contradiction.

Let $m_3$ be an integer such that $m_3 > \lambda \mu+2$.   
Thus for any odd integer $m\geq m_3$ we have that $A_{i+1}$ satisfies property P3(i+1). 
In particular  $\dim_O (A_{i+1}\cap K)  =d$. 
\medskip

Finally, if we let $m_s =\max \{M, m_3\}$, then  for any odd integer $m\geq m_s$ 
we have that $A_{i+1}$ satisfies all the properties P1(i+1), P2(i+1) and P3(i+1).

\bigskip

At the end of the recursive construction, observe that 
the set $A_q$ is algebraic, $A_q \seq A$, $A_q \cap K \seq A\cap K \seq A$ 
and $\dim_O (A_q\cap K) =d$.
Moreover 
$$\ol{A_q \cap K }^Z \leq_s A_q \leq_s  A    \leq_s A_q \cap K \leq_s 
\ol{A_q \cap K }^Z.$$
Thus $S= \ol{A_q \cap K }^Z$ satisfies the thesis.
\end{proof} 

The previous theorem allows us to strengthen the following  result on approximation
preserving dimension which can be found in \cite{FFW-semi}:

\begin{teo}\label{general-semi} 
Let $A$ be a closed semianalytic subset of $\R^n$ with $O\in A$. 
Then for any $s \geq 1$ there exists a closed semialgebraic set $S\subseteq \R^n$ 
such that 
$A \sim_s S$ and $\dim_O S= \dim_O A$.
\end{teo}

From Theorem \ref{preserve-dim} and from Theorem \ref{general-semi} 
we immediately obtain:

\begin{cor}\label{preserve-semianal-dim}
 For any real number
$s\geq 1$ and for any closed semianalytic set $A \subset \R^n$ of
codimension $\geq 1$ with $O\in A$, there exists an algebraic subset
$S$ of $\R^n$ such that $A \seq S$ and $\dim_O S= \dim_O A$.
\end{cor}

\begin{ex}\label{examp}
{\rm If $A=\{(x,y,z)\in \R^3 \ |\ z=0, x\ge 0, y\geq 0\}$ and $s\geq 1$, 
the approximation technique 
described in the proof of Theorem \ref{preserve-dim} yields a surface defined by 
$(z^2-x^m)^2 -y^p=0$ for suitable odd integers $m$ and $p$; the shape of such a 
surface is represented in Figure \ref{fig:app}
\begin{figure}[htbp]
  \centering
  \includegraphics[width=.65\hsize]{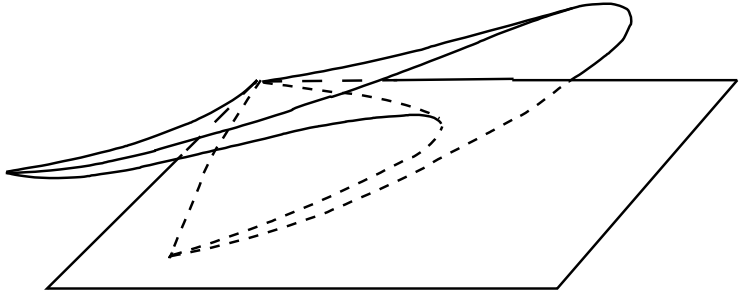} 
  \caption{\footnotesize \sl Algebraic approximation of a quadrant}
  \label{fig:app}
\end{figure}
}
\end{ex}

\end{document}